\documentclass[12pt]{amsart}
\usepackage{latexsym,amsxtra,amscd,ifthen,amsmath,color, multicol,hyperref,mathdots,mathrsfs}
\usepackage{amsfonts}
\usepackage{verbatim}
\usepackage{amsmath}
\usepackage{amsthm}
\usepackage{amssymb}
\usepackage[all,cmtip]{xy}
\usepackage[normalem]{ulem}
\usepackage{enumerate}
\usepackage[latin1]{inputenc}
\usepackage{graphicx}
\usepackage{multirow}
\usepackage{latexsym}
\swapnumbers
\theoremstyle{plain}
\newtheorem{thm}{Theorem}[section]
\newtheorem{lem}[thm]{Lemma}
\newtheorem{prop}[thm]{Proposition}
\newtheorem{cor}[thm]{Corollary}

\theoremstyle{definition}

\newtheorem*{Ack}{Acknowledgement}
\newtheorem{pq-classification}[thm]{Classification of $pq$-dimensional pointed Hopf algebras}

\theoremstyle{remark}

\def\k{\ensuremath{\bold{k}}}

\newcommand*{\e}{\ensuremath{\varepsilon}}

\newcommand*{\field}{\ensuremath{\bold{k}}}

\newcommand*{\gr}{\ensuremath{\text{\upshape gr}}}

\def\Chara{\operatorname{char}}

\begin{document}

\title{FS-indicators of $pq$-dimensional pointed Hopf algebras}

\author{Si Chen}
%\email{sic26@pitt.edu}
\author{Tiantian Liu}
%\email{til43@pitt.edu}
\author{Linhong Wang}
\email{sic26@pitt.edu, til43@pitt.edu, lhwang@pitt.edu}
\address{Department of Mathematics\\
University of Pittsburgh, Pittsburgh, PA 15260}

\author{Xingting Wang}
\email{xingting.wang@howard.edu}
\address{Department of Mathematics\\Howard University, Washington, DC 20059}

\begin{abstract}
We compute higher Frobenius-Schur indicators of $pq$-dimensional pointed Hopf algebras in characteristic $p$ through their associated graded Hopf algebras. These indicators are gauge invariants for the monoidal categories of representations of these algebras.
\end{abstract}

%\thanks{}

\maketitle

\section{Introduction}

Higher Frobenius-Schur (FS) indicators of Hopf algebras were introduced in \cite{LinMon} and were further studied in \cite{KMN, KSZ2, Shi} and other research works. These indicators provide gauge invariants of the monoidal representation categories of these Hopf algebras, see \cite[Theorem 2.2]{KMN}. The values of FS-indicators of certain Hopf algebras have been computed in, e.g., \cite{HHWW, KMN, WW}.  In this short note, we concern the indicators of pointed Hopf algebras of dimension $pq$ over a field of positive characteristic. The classification of such Hopf algebras was recently obtained in \cite{Xio}. Let $r$ be a prime. We define a character function $\chi_r(n)$, associated to the prime $r$, in positive integers $n$ such that 
\[ \chi_r(n)=
\begin{cases}
r, & \text{if} \quad r\mid n\\
1, & \text{if}\quad r\nmid n.
\end{cases}
\]
Then our main result states as follows.
\begin{thm}\label{main}
Let $q\neq p$ be primes, $H$ be a $pq$-dimensional pointed Hopf algebra over characteristic $p$. If $H$ is commutative and cocommutative, then the $n$-th indicator \[\nu_n(H) \equiv \chi_p(n)\chi_q(n) \pmod p.\]
Otherwise, 
\[\nu_n(H)\equiv \chi_p(n) \pmod p.\]
\end{thm}

In this section, we also set up notation and terminology. Readers unfamiliar with Hopf algrebras are referred to standard textbooks such as \cite{Mon, S}. The proof of the main result is in section 2. Throughout, $H$ is a finite-dimensional Hopf algebra over a base field $\k$.
We use the standard notation $(H,\,m,\,u,\,\Delta,\,\e,\,S)$, where $m: H\otimes H \to H$ is the multiplication map, $u: \k\to H$ is the unit map, $\Delta: H\to H\otimes H$ is the comultiplication map, $\e: H\to \k$ is the counit map, and $S: H\to H$ is the antipode. The vector space dual of $H$ is also a Hopf algebra and will be denoted by $H^*$.  We use the Sweedler notation $\Delta(h)=\sum h_{(1)}\otimes h_{(2)}$. If $f,\,g\in H^*$, then $fg(h)=\sum f\left(h_{(1)}\right)g\left(h_{(2)}\right)$ for any $h\in H$ and $\e_{H^*}(f)=f(1)$.

%The \emph{coradical} $H_0$ of $H$ is the sum of all simple subcoalgebras of $H$. If $H_0$ is one-dimensional, $H$ is said to be \emph{connected}. If every simple subcoalgebra is one-dimensional of the form $kg$ where $g$ is a group-like element, $H$ is said \emph{pointed}. The \emph{coradical filtration} $\{H_n\}_{n\ge 0}$ of $H$ is defined inductively such that $H_n=\Delta^{-1}(H\otimes H_{n-1}+H_0\otimes H)$ for all integers $n\ge 1$. We use $\gr_C H=\bigoplus_{i\ge 0} H_{i}/H_{i-1}$ ($H_{-1}=0$) to denote the \emph{associated graded coalgebra} with respect to the coradical filtration of $H$. If $H$ is finite dimensional, then $\gr_C H$ is a graded Hopf algebra if and only if $H_0$ is a Hopf subalgebra. 

%Kashina-Montgomery-Ng  proposed the following definition of higher Frobenius-Schur (FS) indicators for a finite-dimensional Hopf algebra $H$. 

Let $n$ be a positive integer. Suppose $h_1, \ldots$, $h_n \in H$. Then the $n$-th \emph{power of multiplication} is defined as $m^{[n]}(h_1\otimes \cdots \otimes h_n)=h_1\cdots h_n$.
Let $h\in H$. The $n$-th \emph{\emph{power of comultiplication}} is defined to be
\[
\Delta^{[n]}(h)=\begin{cases}
h & n=1\\
(\Delta^{[n-1]}\otimes \mathrm{id})\left(\Delta(h)\right) & n\geq 2
\end{cases}
\]
The $n$-th \emph{Sweedler power} of $h$ is defined to be \[P_n(h)=h^{[n]}=
\begin{cases}
\epsilon(h) 1_H & n=0\\
m^{[n]}\circ\Delta^{[n]}(h) & n\geq 1
\end{cases}
\]
Then the $n$-th \emph{indicator} \cite[Definition 2.1]{KMN} of $H$ is given by  $$\nu_n(H)=\mathrm{Tr}\big(S\circ P_{n-1}\big).$$ In particular, $\nu_1(H)=1$ and $\nu_2(H)=\mathrm{Tr}(S)$. 

%It was proved in \cite[Theorem 2.2]{KMN} that the sequence $\{\nu_n(H)\}$ is an invariant of the gauge equivalence class of Hopf algebras of $H$, that is, if $H$ and $K$ are gauge equivalent then $\{\nu_n(H)\}=\{\nu_n(K)\}$.

\begin{Ack}
We began this work in an undergraduate research project at the University of Pittsburgh. We would like to express our gratitude to the math department for its supporting.
\end{Ack}

%%%%%%%%%%%%%%%%%%%%%%%%%%%%%%%%%%%%%%%%%%%%%%%%%%%%%%%%%%%%%%%%

\section{The proof of the main result}
In this section, $q\neq p$ are primes, the base field $\field$ is algebraically closed of characteristic $p$, and $\xi$ is a $q$th primitive root of unity.

\begin{thm}\cite[Theorem 2.19]{Xio}\label{pqClass}
Let $\delta=0$ or $1$. Then a $pq$-dimensional pointed Hopf algebra $H$  is isomorphic to one of the following
\begin{itemize}
\item [1)]
$\field[g,\, x]/(g^q-1,\  x^p-\delta x)$,
\item [] $\Delta(g)=g\otimes g, \Delta(x)=x\otimes 1 + 1\otimes x$;
\vspace{0.1in}

\item [2)]
$\begin{cases}
\field\langle g,\, x\rangle /(g^q-1,\  x^p,\ xg-\xi gx),& q\nmid p-1,\\
\field\langle g,\, x\rangle /(g^q-1,\  x^p-\delta x,\ gx-\xi xg),& q\mid p-1,
\end{cases}$
\item [] $\Delta(g)=g\otimes g, \Delta(x)=x\otimes 1 + 1\otimes x$;
\vspace{0.1in}

\item [3)]
$\field[g,\, x]/(g^q-1,\  x^p-\delta (1-g^p))$,
\item [] $\Delta(g)=g\otimes g, \Delta(x)=x\otimes 1 + g\otimes x$;
\vspace{0.1in}

\item [4)]
$\begin{cases}
\field\langle g,\, x\rangle /(g^q-1,\  x^p-x,\ gx- xg-g+g^2), & q\nmid p-1,\\
\field [g,\, x] /(g^q-1,\  x^p-x), & q\mid p-1,
\end{cases}$
\item [] $\Delta(g)=g\otimes g, \Delta(x)=x\otimes 1 + g\otimes x.$
\end{itemize}
\end{thm}

The approach of our proof is based on the above classification result and the following theorem and proposition. We also provide a lemma prior to the proof of Theorem \ref{main}. A \emph{left integral} in a finite dimensional Hopf algebra $H$ is an element $\Lambda\in H$ such that $h\Lambda=\e(h)\Lambda$, for all $h\in H$. The span of the left integrals in $H$ is a one-dimensional space. 

\begin{thm}\cite[Corollary 2.6]{KMN}\label{T:FS}
 Suppose $\lambda\in H^*$ and $\Lambda\in H$ are both left integrals %(or both right integrals) 
 such that $\lambda(\Lambda)=1$. Then
$
\nu_n(H)=\lambda \left(\Lambda^{[n]}\right)
$
for all positive integers $n$.
\end{thm}

%Recently, Shimizu \cite{Shi} investigated more properties of these indicators in \cite{Shi}.
\begin{prop}\cite[Corollary 3.17, Corollary 4.6]{Shi}\label{PIN}
Let $H$ and $H'$ be finite-dimensional Hopf algebras. Then, for all positive integers $n$, we have the following:
\begin{itemize}
\item[1)] $\nu_n(H^*)=\nu_n(H)$ and $\nu_n(H\otimes H')=\nu_n(H)\cdot \nu_n(H')$.
%\item[2)] $\nu_n(H)=\nu_n(H\otimes K)$, where $K$ is a field extension of $k$.
\item[2)] If $H$ is filtered, then $\nu_n(\gr H)=\nu_n(H)$.
\item[3)] If $\Chara \k >0$, then $\{\nu_n(H)\}$ is periodic.
\end{itemize}
\end{prop}

Although Proposition \ref{PIN}, 2) is true for any Hopf filtration, in this paper we use coradical filtration. 

%The \emph{coradical} $H_0$ of $H$ is the sum of all simple subcoalgebras of $H$. If $H_0$ is one-dimensional, $H$ is said to be \emph{connected}. If every simple subcoalgebra is one-dimensional of the form $kg$ where $g$ is a group-like element, $H$ is said \emph{pointed}. The \emph{coradical filtration} $\{H_n\}_{n\ge 0}$ of $H$ is defined inductively such that $H_n=\Delta^{-1}(H\otimes H_{n-1}+H_0\otimes H)$ for all integers $n\ge 1$. We use $\gr_C H=\bigoplus_{i\ge 0} H_{i}/H_{i-1}$ ($H_{-1}=0$) to denote the \emph{associated graded coalgebra} with respect to the coradical filtration of $H$. If $H$ is finite dimensional, then $\gr_C H$ is a graded Hopf algebra if and only if $H_0$ is a Hopf subalgebra. 

\begin{lem}\label{nplus1}
Let $i, n$ be integers for $0\leq i\leq q-1$ and $n\geq 1$. Suppose nonnegative integers $k_1,\ldots,k_{n+1}$ form a partition of $p-1$.
\begin{itemize}
\item[1)] If $xg=gx$, $\Delta(g)=g\otimes g$, and $\Delta(x)=x\otimes 1 + g\otimes x$, then
\[
(g^i x^{p-1})^{[n+1]}=\sum_{\substack{0\le k_1,\ldots,k_n\le p-1}} {p-1\choose k_1,\ldots,k_{n+1}}
g^{(n+1)i+nk_1+(n-1)k_2+\cdots+k_n}x^{p-1}.
\]
\item[2)] If $xg=\xi gx$, $\Delta(g)=g\otimes g$, and $\Delta(x)=x\otimes 1 + 1\otimes x$, then
\[(g^i x^{p-1})^{[n+1]}=\bigg( (\xi^i)^n+\cdots+\xi^i+1 \bigg)^{p-1}g^{(n+1)i}x^{p-1}.
\]
\end{itemize}
\end{lem}

\begin{proof}
1) If $n=1$, then $(g^i x^{p-1})^{[n+1]}=(g^i x^{p-1})^{[2]}=m\circ \Delta (g^i x^{p-1})$. Firstly,
%=\Delta(g)^i\Delta(x)^{p-1}
\[
\Delta(g^ix^{p-1})=g^i\otimes g^i (x\otimes 1+g\otimes x)^{p-1}=\sum_{k_1=0}^{p-1} {p-1 \choose k_1}g^ix^{p-1-k_1}g^{k_1}\otimes g^ix^{k_1}.
\]
%\begin{align*}
%\Delta(g^ix^{p-1})&=\Delta(g)^i\Delta(x)^{p-1}=g^i\otimes g^i (x\otimes 1+g\otimes x)^{p-1}\\
%&=\sum_{k_1=0}^{p-1} {p-1 \choose k_1}g^ix^{p-1-k_1}g^{k_1}\otimes g^ix^{k_1}
%\end{align*}
Note that $xg=gx$. Then
\[
(g^i x^{p-1})^{[2]}=\sum_{k_1=0}^{p-1} {p-1 \choose k_1} g^{2i+k_1}x^{p-1}.
\]
Now let $n=2$, and $k_1, k_2, k_3$ form a partition of $p-1$. Then
\[
\Delta^{[3]}(g^ix^{p-1})=\Delta\otimes \text{id}\circ \Delta(g^ix^{p-1})
=\sum_{k_1=0}^{p-1} {p-1 \choose k_1}\Delta(g^ix^{p-1-k_1})\Delta(g)^{k_1}\otimes \text{id}(g^ix^{k_1}).
\]
Following the case when $n=1$, we have
\[
\Delta(g^i x^{p-1-k_1})=\sum_{k_2=0}^{p-1-k_1} {p-1-k_1 \choose k_2}g^ix^{k_3}g^{k_2}\otimes g^ix^{k_1}.
\]
Hence,
\[
\Delta^{[3]}(g^ix^{p-1})=\sum_{k_1=0}^{p-1}\sum_{k_2=0}^{p-1-k_1} {p-1\choose k_1} {p-1-k_1 \choose k_2}g^ix^{k_3}g^{k_1+k_2}\otimes g^ix^{k_2}g^{k_1}\otimes g^ix^{k_1}.
\]
Thus, using multinomial coefficients, we have
\begin{align*}
(g^ix^{p-1})^{[3]}&
=\sum_{0\leq k_1, k_2\leq p-1}^{p-1}{p-1\choose k_1, k_2, k_3} g^{3i+k_1+k_2+k_1}x^{p-1}\\
&=\sum_{0\leq k_1, k_2\leq p-1}^{p-1}{p-1\choose k_1, k_2, k_3} g^{3i+k_1+2k_2}x^{p-1}.
\end{align*}
Finally, one can show inductively that
\[
(g^i x^{p-1})^{[n+1]}=\sum_{\substack{0\le k_1,\ldots,k_n\le p-1}} {p-1\choose k_1,\ldots,k_{n+1}}
g^{(n+1)i+k_1+2k_2+\cdots+nk_n}x^{p-1}.
\]

2) It can be shown that
\begin{align*}
\Delta^{[n+1]}(x^{p-1})&=\big(\Delta^{[n+1]}(x)\big)^{p-1}=\big(x\otimes 1\otimes\ldots\otimes 1+\ldots + 1\otimes\ldots\otimes 1\otimes x\big)^{p-1}\\
&=\sum_{\substack{0\le k_1,\ldots,k_n\le p-1}} {p-1\choose k_1,\ldots,k_{n+1}}
x^{k_1}\otimes x^{k_2}\ldots \otimes x^{k_{n+1}}.
\end{align*}
Note that we have $xg=\xi gx$. Hence,
\begin{align*}
&(g^ix^{p-1})^{[n+1]}=\sum_{\substack{0\le k_1,\ldots,k_n\le p-1}} {p-1\choose k_1,\ldots,k_{n+1}}g^ix^{k_1} g^ix^{k_2}\ldots g^ix^{k_{n+1}}\\
=&\sum_{\substack{0\le k_1,\ldots,k_n\le p-1}} {p-1\choose k_1,\ldots,k_{n+1}}
(\xi^i)^{nk_1}(\xi^i)^{(n-1)k_2}\cdot(\xi^i)^{k_n}\cdot g^{(n+1)i}x^{p-1}\\
=&\bigg((\xi^i)^n+\cdots+\xi^i+1 \bigg)^{p-1} g^{(n+1)i}x^{p-1}.
\end{align*}
\end{proof}

\noindent\textbf{The proof of Theorem \ref{main}.}
Let $H$ be a $pq$-dimensional pointed Hopf algebra $H$ over $\field$ given in Theorem \ref{pqClass}. The linear basis of $H$ can be chosen as $\{g^ix^j\, |\, 0\le i \le p-1, 0\le j \le q-1\}$. The associated graded Hopf algebras $\gr H$ is one of the following, in the correspondence $1)\rightsquigarrow  A$, $3), 4)\rightsquigarrow B$, and $2)\rightsquigarrow C$.
\begin{itemize}
\item [] $A=\field[g,x]/ (g^q-1,\  x^p)$, with $\Delta(g)=g\otimes g, \Delta(x)=x\otimes 1 + 1\otimes x$.
\vspace{0.1in}

\item [] $B=\field[g,x]/ (g^q-1,\  x^p)$, with $\Delta(g)=g\otimes g, \Delta(x)=x\otimes 1 + g\otimes x$.
\vspace{0.1in}

\item [] $C=\field \langle g,x\rangle / (g^q-1,\  x^p,\ xg-\xi gx)$, with $\Delta(g)=g\otimes g, \Delta(x)=x\otimes 1 + 1\otimes x$.
\end{itemize}
Next we find case by case the indicators of these associated graded Hopf algebras.

\textbf{Case $A$}. Note that $A=\field[g]/(g^q-1)\otimes \field[x]/(x^p)$ with $\Delta(g)=g\otimes g$, and  $\Delta(x)=x\otimes 1 + 1\otimes x$.
Following from the counting argument
\[\nu_n(\k[g]/(g^q-1))=\#\{h \in \langle g \rangle\, |\, h^n=1_{\langle g\rangle} \},\]
we have \[ \nu_n\big(\field[g]/(g^q-1)\big)= \chi_q(n), \quad \text{for all}\quad n\geq 1.\]
By \cite[Corollary 3.9]{WW}, it follows that
\[\nu_n\big(\field[x]/(x^p)\big)\equiv \chi_p(n) \pmod p, \quad \text{for all}\quad n\geq 1.\]
Therefore \[\nu_n(A)\equiv \chi_p(n)\chi_q(n) \pmod p, \quad \text{for all}\quad n\geq 1. \]

%%%%%%%%%%%%%%%%%%%%%%Case 2 %%%%%%%%%%%%%%%%%%%%
\textbf{Case $B$}. The comultiplication on  the basis elements is given by
%$g^ix^j$ for  $0\leq i\leq q-1$ and  $0\leq j\leq p-1$. We have
\[\Delta(g^ix^j)=\sum_{k=0}^j{j \choose k} g^{k+i}x^{j-k}\otimes g^ix^k.\]
Let $\lambda_B= \delta_{g^{1-p}x^{p-1}} \in B^*$. Then
\[
\delta_{g^{1-p}x^{p-1}} (g^i x^j)=
\begin{cases}
1 &  \text{if}\; j=p-1\; \text{and}\;  i\equiv 1-p \pmod q\\
0 & \text{otherwise}.
\end{cases}
\]
On the other hand,
\begin{align*}
\sum_{k=0}^j  &{j \choose k}g^{k+i}x^{j-k}\cdot \delta_{g^{1-p}x^{p-1}} (g^ix^k) \\
&=\begin{cases}
g^{p-1+1-p}x^0=1 & \text{if}\; j=k=p-1\; \text{and}\;  i\equiv 1-p \pmod q\\
0 & \text{otherwise}.
\end{cases}
\end{align*}
Therefore, by \cite[Lemma 2.3]{HHWW},
%that $\lambda$ is a left integral of $H^*$ if and only if $\sum h_{(1)}\lambda (h_{(2)})=\lambda (h)$ for any $h\in H$.
$\lambda_B= \delta_{g^{1-p}x^{p-1}}$ is a left integral of  $B^*$.
Next consider the element $\Lambda_B=(1+g+\cdots+g^{q-1}) x^{p-1}\in B$. It is clear that $x\Lambda_B=0=\e(x)\Lambda_B$ and $g\Lambda_B=\Lambda_B=\e(g)\Lambda_B$. Hence $\Lambda_B=(1+g+\cdots+g^{q-1})x^{p-1}$ is a left integral of $B$. It is clear that $\lambda_B\Big(\Lambda_B)=1$ . Following Theorem \ref{T:FS}, for $n\geq 1$ we compute
\[
\nu_{n+1}(B)=\lambda_B\Big(\Lambda_B^{[n+1]}\Big)=  \delta_{g^{1-p}x^{p-1}}\Big(\sum_{i=0}^{q-1} (g^ix^{p-1})^{[n+1]}\Big).
\]
Note that $gx=xg$ and $\delta(x)=x\otimes 1+ g\otimes x$. By Lemma \ref{nplus1},
\[
\nu_{n+1}(B)= \sum_{\substack{0\le i\le q-1\\ 0\le k_1,\ldots,k_n\le p-1}} {p-1\choose k_1,\ldots,k_{n+1}} \delta_{g^{1-p}}
\big(g^{(n+1)i+k_1+2k_2+\cdots+nk_n}\big).
\]
For any fixed partition $k_1, \ldots, k_n, k_{n+1}$ of $p-1$ we have
\[
\sum_{i=0}^{q-1} \delta_{g^{1-p}}
\big(g^{(n+1)i+k_1+2k_2+\cdots+nk_n}\big)
=\begin{cases}
1 & \text{if}\; q \nmid n+1\\
q  \delta_{g^{1-p}}
\big(g^{k_1+2k_2+\cdots+nk_n}\big) & \text{if}\; q \mid n+1
\end{cases}.
\]
Hence, for $q \nmid n+1$, we have
\[
\nu_{n+1}(B)= \sum_{\substack{0\le i\le q-1\\ 0\le k_1,\ldots,k_n\le p-1}} {p-1\choose k_1,\ldots,k_{n+1}}=(n+1)^{p-1}.
\]
For $q\mid n+1$, we have
\[
\nu_{n+1}(B)= \sum_{\substack{ k_1+2k_2+\cdots+nk_n \equiv 1-p \pmod q \\ 0\le k_1,\ldots,k_n\le p-1}}  q {p-1\choose k_1,\ldots,k_{n+1}}.
\]
%$\nu_{1}(B)=1$
%%%%%%%%%%%%%%%%%%%%%%Case 3 %%%%%%%%%%%%%%%%%%%%
\textbf{Case $C$}. Since $\Delta(x)=x\otimes 1+ 1\otimes x$, we have
\[\Delta(g^ix^j)=\sum_{k=0}^j{j \choose k} g^ix^{j-k}\otimes g^ix^k,\]
for the basis elements $g^ix^j$ for  $0\leq i\leq q-1$ and  $0\leq j\leq p-1$. By a similar argument as that for $B$, the element $\lambda_C=\delta_{x^{p-1}} \in C^{*}$  is an left integral of $C^{*}$. Next consider the element $\Lambda_C=(1+g+\cdots+g^{q-1}) x^{p-1}\in C$. It is clear that $g\Lambda_C=\e(g)\Lambda_C$. Note that $xg=\xi gx$. Then
\[
x\Lambda_C=x(1+g+\cdots+g^{q-1}) x^{p-1}= \big(1+ \xi g+ \cdots + (\xi g)^{q-1}\big)x^p= 0=\e(x)\Lambda_C.
\]
Hence $\Lambda_C=(1+g+\cdots+g^{q-1})x^{p-1}$ is a left integral of $C$.
Again, by Theorem \ref{T:FS},
\begin{align*}
\nu_{n+1}(C)&
=  \delta_{x^{p-1}}\bigg(\sum_{i=0}^{q-1} (g^ix^{p-1})^{[n+1]}\bigg)=\delta_{x^{p-1}}\bigg(\sum_{i=0}^{q-1}  \big( (\xi^i)^n+\cdots+\xi^i+1 \big)^{p-1}g^{(n+1)i}\bigg)\\
&=\delta_{x^{p-1}}\bigg((n+1)^{p-1}x^{p-1} + \sum_{i=1}^{q-1}  \frac{(\xi^i)^{n+1}-1}{\xi^i-1}g^{(n+1)i}x^{p-1}\bigg)=(n+1)^{p-1}.
\end{align*}
It follows from the Little Fermat's Theorem that  $(n)^{p-1} \equiv \chi_p(n) \pmod p$. Hence we have $\nu_{n}(C)= \chi_p(n)$ for any $n\geq 1$. 

Note that $A$ is self-dual, $B$ and $C$ are dual Hopf algebras.  Therefore, combining Proposition \ref{PIN} and the three cases, we prove Theorem \ref{main} and the following corollary.

%Since $B$ and $C$ are dual Hopf algebras, $\nu_{n+1}(B)=\nu_{n+1}(C)$ for all $n\geq 1$. We obtain 

\begin{cor}
The following combinatorics identity holds.
\[
\sum_{\substack{ k_1+2k_2+\cdots+(n-1)k_{n-1} \equiv 1-p \\
\pmod q, \text{ for } 0\le k_1,\ldots,k_{n-1}\le p-1}}  q {p-1\choose k_1,\ldots,k_{n}}
=n^{p-1}\; \pmod p.
\]
where $p$, $q$ are primes, $n> 1$ is an integer, and $q\mid n$.
\end{cor}

%%%%%%%%%%%%%%%%%%%%%%%%%%%%%%%%%%

\end{document}